\documentclass[12pt,a4paper]{article}

\usepackage[latin1]{inputenc}
\usepackage[english]{babel}
\usepackage{amsmath}
\usepackage{color}
\usepackage{amsfonts,enumitem}
\usepackage{amssymb}
\usepackage{hyperref} 
\usepackage{graphicx,tikz}

\usepackage{algorithmic}
\usepackage[ruled,noend]{algorithm2e}
\usepackage[skins]{tcolorbox}

\newcommand{\conv}{\mathrm{conv}}
\newcommand{\proj}{\mathrm{proj}}
\newcommand{\dist}{\mathrm{dist}}
\newcommand{\cl}{\mathrm{cl}\,}

\newcommand{\QQ}{\mathbb{Q}}

\newcommand{\SSS}{\mathcal{S}}

\newcommand{\RR}{\mathbb{R}}
\newcommand{\NN}{\mathbb{N}}
\newcommand{\dd}{{\rm d}}

\newcommand{\partialc}{\partial^c}

\newtheorem{theorem}{Theorem}
\newtheorem{proposition}[theorem]{Proposition}
\newtheorem{lemma}[theorem]{Lemma}
\newtheorem{corollary}[theorem]{Corollary}

\newtheorem{example}{Example}
\newtheorem{remark}{Remark}

\newtheorem{definition}{Definition}%

\newenvironment{proof}[1][]{\noindent {\bf Proof #1:\;}}{\hfill $\Box$}

\providecommand{\keywords}[1]{\textbf{\textbf{Keywords. }} #1}

\newcommand{\jac}{\mathrm{Jac}\,}

\newcommand{\alg}{$\mathrm{RM}$}

\title{Conservative parametric optimality and the ridge method for tame min-max problems}
\begin{document}

\author{Edouard Pauwels\thanks{IRIT, Universit\'e de Toulouse, CNRS, Institut Universitaire de France (IUF). France.}}

\date{\today}

\maketitle

\begin{abstract}
	
				We study the ridge method for min-max problems, and investigate its convergence without any convexity, differentiability or qualification assumption. 
				The central issue is to determine whether the ``parametric optimality formula'' provides a  conservative gradient, a notion of generalized derivative well suited for optimization. The answer to this question is positive in a semi-algebraic, and more generally definable, context. As a consequence, the ridge method applied to definable objectives is proved to have a minimizing behavior and to converge to a set of equilibria which satisfy an optimality condition. Definability is key to our proof: we show that for a more general class of nonsmooth functions, conservativity of the parametric optimality formula may fail, resulting in an absurd behavior of the ridge method.
\end{abstract}

\section{Introduction}
\subsection{Main result}
\label{sec:introduction}
We consider unconstrained minimization of an objective function:
\begin{align}
				f \colon x &\to \max_{y \in \RR^r} F(x,y),
				\label{eq:valueFunction}
\end{align}
where $F \colon \RR^p \times \RR^r \to \RR$ is locally Lipschitz and achieves its maximum \footnote{The maximum in \eqref{eq:valueFunction} is arbitrary and could be reversed into a minimum.} in $y$ for all $x \in \RR^p$ such that the argmax exists in a locally bounded set. The function $f$ will be called the value function, note that in this situation, $f$ is also locally Lipschitz. These notations and assumptions will be standing throughout this paper. We consider the Ridge Method (RM)\footnote{We are not aware of detailed description of this method in the literature, although the idea existed in discussions. We chose the name "Ridge Method" from oral transmition.}, initialized with $x_0 \in \RR^p$, it is defined recursively as follows, for all $k \in \NN$,
\begin{align}
				y_k &\in \arg\max_{y \in \RR^r} F(x_k,y) \nonumber\\
				(u_k, 0) &\in \partialc F(x_k, y_k)\tag{\alg} \label{eq:ridgeAlgorithm}\\
				x_{k+1} &= x_k - \alpha_k u_k\nonumber,
\end{align}
where $(\alpha_k)_{k\in \NN}$ is a nonsummable sequence of positive step sizes tending to $0$ and $\partialc$ denotes the Clarke subdifferential \cite{clarke1983optimization}.
The existence of the update direction $u_k$ in \eqref{eq:ridgeAlgorithm} is ensured by the Parametric Optimality (PO) formula for subgradient of partial maxima. The algorithm relies on the knowledge of
\begin{itemize}
				\item {\bf A partial maximization oracle}, which associates to any $x \in \RR^p$, an element of the set $P(x) := {\arg\max}_{y \in \RR^r} F(x,y)$, assumed to be nonempty.
				\item {\bf A first order oracle}, which associates to any $(x,y) \in \RR^p \times \RR^r$ first order information about $F$ in the form of its Clarke subdifferential, $\partialc F$.
\end{itemize}

For most locally Lipschitz functions $F$, the Clarke subdifferential, $\partialc F$ \cite{clarke1983optimization} carries absolutely no information about the function itself and is therefore useless from a computational perspective \cite{wang1995pathological,borwein2001generalized,borwein2017genralisations}. For our purpose, we need to restrict $F$ to be in a subclass which is well behaved with respect to subdifferentiation, we choose the class of path differentiable functions, which was identified by several authors to be well behaved in terms of subgradient differential inclusion \cite{valadier1989entrainement,borwein1998chain,davis2018stochastic,bolte2020conservative}. 

The purpose of this paper is to investigate asymptotic behavior of \eqref{eq:ridgeAlgorithm}. The main results are the following.
\begin{itemize}
				\item When applied to a large and widespread subclass of functions $F$, for example semi-algebraic functions, and more generally definable functions, Algorithm \eqref{eq:ridgeAlgorithm} has a minimizing behavior. For bounded sequences, the value function, $f(x_k)$, is converging and accumulation points, $(\bar{x}, \bar{y})$, of the sequence $(x_k,y_k)_{k \in \NN}$ are equilibria which satisfy the optimality condition:
				\begin{align*}
								0 &\in \mathrm{conv} \{u,\, (u,0) \in \partialc F(\bar{x},y), \, y \in \arg\max_{z \in \RR^r} F(\bar{x},z)\} \\
								\bar{y} &\in \arg\max_{y \in \RR^r} F(\bar{x},y)
				\end{align*}
				Let us stress that most objectives found in applications are definable, see \cite{attouch2013convergence,bolte2014proximal} for discussion and examples, and \cite{bolte2020mathematical,bolte2020conservative} for a recent account in deep learning. 
				\item Without definability assumption, for general path-differentiable functions, the algorithm may fail to have any minimizing property. We construct a Lipschitz function $F$ which is path-differentiable, such that all inputs $x$ are steady states for \eqref{eq:ridgeAlgorithm} but not critical in any reasonable sense for $f$. This underlines the importance of the definability assumption in the previous result and shows that Algorithm \eqref{eq:ridgeAlgorithm} requires to work with proper subclasses, beyond Lipschicity and path-differentiability.
\end{itemize}

\subsection{Parametric optimality and nonsmooth differential calculus}
In order to analyse algorithm \eqref{eq:ridgeAlgorithm}, we need a variational model for parametric optimality and calculus rules providing access to first order information for $f$, from the knowledge of the partial maximization oracle and the subgradient of $F$. Our main candidate is the parametric optimality formula, which was described in \cite[Corollary 3.1.1]{rockafellar1985extensions} for partial minimization when $F$ is level bounded in $y$ locally uniformly in $x$, see also \cite[Theorem 10.13]{rockafellar1998variational}. Clarke subdifferential is non directional, and therefore, the formula is also valid for partial maximization. Under our setting, it ensures that for all $x \in \RR^p$
\begin{align}
				\partialc f(x)	\subset \mathrm{conv} \{u,\, (u,0) \in \partialc F(x,y), \, y \in P(x)\}. \tag{PO}
				\label{eq:partialOptimality}
\end{align}
The formula is in fact a result describing the subgradient of $f$, but we will be interested mostly in its right hand side which will be referred to as the parametric optimality formula (PO formula)
\footnote{\textbf{Envelope formula:} the PO formula is very similar to the envelope formula: $\partialc f(x)	\subset \mathrm{conv} \{u \in \partialc_x F(x,y), \, y \in P(x)\}$, \cite[Theorem 2.8.2]{clarke1983optimization} where $\partialc_x$ denotes the subgradient for fixed $y$. This formula holds with equality if $F$ is convex in $x$, see\cite[Proposition A.22]{bertsekas1971control}\cite{demyanov1966solution,danskin1966theory}, and more generally regular \cite[Theorem 2.8.2]{clarke1983optimization}. This is simpler and computationally more advantageous than the PO formula. However, it very much depends on convexity and is too coarse in general for our purposes, consider for example $F(x,y) = -\vert x-y\vert$, the envelope formula gives $[-1,1]$ for all $x$ and $f(x) = \max_y F(x,y) = 0$ for all $x$.}. 
Indeed, any element given by the PO formula can be computed from the knowledge of maximization and first order oracles mentioned above and precisely corresponds to the search direction chosen in Algorithm \eqref{eq:ridgeAlgorithm}.

As noted in \cite[Theorem 10.13]{rockafellar1998variational}, the PO formula is sharp, it holds with equality in the case of a concave $F$ (jointly in $(x,y)$).
However in general nonconvex settings, without further qualification assumptions, the PO formula does not hold with equality\footnote{
\textbf{Failure of partial optimality formula:} 
\label{ex:failure1}
Consider the function $F \colon (x,y) \mapsto  -y \min\{ \vert x\vert,1\} + \min\{0,y\}$, which is concave in $y$.
We have for all $x\in \RR$, $f(x) = \max_{y \in \RR} F(x,y)= 0$. Yet fixing $x = 0$, we have $F(0,y) = 0$ for all $y > 0$, all such $y$ being partial maximizers. However $([-y,y], 0) \subset \partialc F(0,y)$ for any $y > 0$. This shows that any $s \in \RR$ is compatible with the PO formula at $x=0$, yet the corresponding Clarke subdifferential is only the singleton $\{0\}$, which coincides with the classical derivative. 
}.
Therefore, the PO formula does not necessarily provide a subgradient, this constitutes the main difficulty in analyzing Algorithm \eqref{eq:ridgeAlgorithm}. We choose to use conservativity, a notion of generalized derivative which was recently introduced as a nonsmooth analysis tool which is compatible with differential calculus \cite{bolte2020conservative}.
Most importantly, conservative gradients may be used in place of subgradients for first order optimization \cite{castera2019inertial,bolte2020conservative,bolte2020mathematical}, making them a natural candidate for algorithmic oracles in our context.
With this in mind, the proposed analysis of Algorithm \eqref{eq:ridgeAlgorithm} boils down to the central question:
\begin{center}
				\emph{For a path-differentiable $F$, does the PO formula define a conservative gradient for $f$?}
\end{center}
The answer to this question is negative in general, we provide a counterexample. However for definable functions the answer turns out to be positive. The proof of the latter result relies on a characterization of definable conservative gradients based only on definable paths, which is of independent interest (see also remark \ref{rem:davis} and \cite{davis2021conservative}). In the context of conservativity, the definable case plays a special role as it is widepread in applications \cite{bolte2020conservative,bolte2020mathematical} and many further properties are available \cite{bolte2020mathematical,lewis2021structure,davis2021conservative}. The reader unfamiliar with definability may consider instead semialgebraicity, which is a special case, a function being semialgebraic when its graph can be represented as the finite union of solution sets of polynomial systems involving finitely many equalities and inequalities. 
Section \ref{sec:presentation} exposes basic definitions and more details regarding definability. 

Let us describe a few applications and consequences of these results:
\begin{itemize}
				\item If $F$ is given in the form of a composition of elementary definable functions, Lipschitz and concave in its second argument, then the maximization oracle could be given by an optimization solver which adresses the maximization problem to global optimality, and the sugradient oracle could be given by algorithmic differentiation (see remark \ref{rem:definableConservativeField}). We ensure that the ridge method \eqref{eq:ridgeAlgorithm} is attracted by stationary points in this case and actually most sequences are attracted by Clarke critical points \cite{bolte2020mathematical,bianchi2020convergence}.
				\item In the previous example, concavity is instrumental, to ensure that global maximization is reasonable. Actually, the only requirement is to have a global maximization oracle, for which we could consider global optimization examples beyond concavity.
				\item In both cases, the \eqref{eq:partialOptimality} formula evaluated at $x$ is a singleton equal to $\nabla f(x)$, the classical gradient, everywhere outside of a finite union of differentiable manifolds of dimension strictly less than $p$ (a negligible set). This is a consequence of \cite[Theorem 1]{bolte2020mathematical} and illustrates the fact that, although the \ref{eq:partialOptimality} formula may produce artifacts, it occurs very rarely.
				\item Finally, conservative gradients satisfy Fermat rule, using Carath\'eodory theorem, this translates in an optimality condition as follows: if $x \in \RR^p$ is a local minimum of $f$, then there exists $(y_i,u_i,\lambda_i)_{i=1}^{p+1}$ such that 
								\begin{align}
												&y_i \in \arg\max_{y \in \RR^r} F(x,y),\qquad (u_i,0) \in \partial^c F(x, y_i), \qquad \lambda_i \in [0,1],\qquad i= 1,\ldots, p+1 \nonumber\\
												&\sum_{i=1}^{p+1} \lambda_i u_i = 0,\qquad \sum_{i=1}^{p+1} \lambda_i = 1.
												\label{eq:optiCondition}
								\end{align}
								This constitutes a notion of stationarity for nonsmooth min-max problems which could be a target for future algorithmic developments.
\end{itemize}

\subsection{Applications of the main results and existing literature}
Min-max problems arise in machine learning applications with Generative Adversarial Networks (GANs) \cite{goodfellow2014generative,arjovsky2017wasserstein}, adversarial training of deep networks \cite{szegedy2013intriguing,goodfellow2015explaining}, and further applications \cite{ablin2020super} bolstering reasearch on algorithms for min-max problems, see for example \cite{wang2020solving} for a ridge-type method. Most of these applications amount to solving nonsmooth min-max problems, with non convex or concave definable objectives, based on first order methods and algorithmic differentiation oracles. Our results describe a notion of equilibrium for such problem as well as an algorithm to reach such equilibria, under assumptions which are general enough to encompass most machine learning applications \cite{bolte2020mathematical}. 

Our main results are actually stated in terms of conservative gradients, a notion which has been shown to be compatible with the rules of differential calculus, contrary to the notion of subdifferential \cite{bolte2020conservative,bolte2020mathematical}. Beyond min-max problems, the fact that the PO formula defines a conservative gradient can be used in compositional modeling, in combination with algorithmic differentiation. This is in close connection with emerging extensions of deep neural networks which include optimization problems within their formulation, some of the network layers being defined as partial maxima or minima \cite{amos2017optnet,agrawal2019differentiable,berthet2020learning}. Conservativity of the PO formula provides theoretical ground to develop algorithmic differentiation tools for compositional problems involving such max structured functions.

Stationarity for min-max problem is an important research topic, a large literature is devoted to convex-concave min-max problems, which we will not describe as we consider a much broader class. In the nonconvex case, it is particularly interesting to restrict $F$ to be twice continuously differentiable and strongly convex in $y$ for each fixed $x$, as the value function $f$ and the argmax mapping are differentiable in this case (using for example the implicit function theorem). Closer to our interest is the situation where the maximization problem is merely convex in $y$ for each fixed $x$. In this case, even if $F$ is smooth, the value function $f$ may not be differentiable. The work of \cite{rafique2021weakly} consider min-max structured problem with $F$ concave in $y$ and weakly convex in $x$. They observe that in this case the value function is weakly convex and propose a dedicated algorithm. It turns out that the envelope formula remains valid with equality for such problems (see for example \cite[Lemma 4.7]{lin2020gradient} which is a special case of \cite[Theorem 2.8.2]{clarke1983optimization}). Since functions with Lipschitz gradient are regular (actually weakly convex) these observations allowed to develop a variety of algorithms and analyses for twice differentiable $F$, concave in its second argument. Most proposed approach rely on equality in the envelope formula and the Moreau envelope of the value function \cite{ostrovskii2021efficient,kong2019accelerated,thekumparampil2019efficient,lin2020gradient,jin2020what,rafique2021weakly}. These provide a precise picture regarding stationarity and algorithms for smooth min max problems with concave maximization component. The present work departs from this literature because it relies neither on smoothness nor on a form of concavity or weak convexity and the obtained results are only qualitative.

\section{Presentation of the main results}
\label{sec:presentation}

\subsection{Technical preliminary}
\subsubsection{Nonsmooth analysis and conservative gradients}
For any integer $p \in \NN$, we use the following notations. We denote by $\left\langle \cdot,\cdot\right\rangle$ be the canonical Euclidean scalar product on $\RR^p$ and $\|\cdot\|$ its associated norm. A locally Lipschitz continuous function, $f \colon \RR^p \to \RR$ is differentiable almost everywhere by Rademacher's theorem, see for example \cite{evans2015measure}. Denote by $R \subset \RR^p$, the full measure set where $f$ is differentiable, then the Clarke subdifferential \cite{clarke1983optimization} of $f$ is given for any $x \in \RR^p$, by
\begin{align*}
				\partialc f(x) = \mathrm{conv} \left\{ v \in \RR^p,\, \exists  y_k \underset{k \to \infty}{\to} x \text{ with } y_k \in R,\, v_k = \nabla f(y_k) \underset{k \to \infty}{\to} v \right\}.
\end{align*}
A set valued map $D\colon \RR^p \rightrightarrows \RR^q$ is a function from $\RR^p$ to  the set of subsets of $\RR^q$. The graph of $D$ is given by
\begin{align*}
				\mathrm{graph}\,D = \left\{ (x,z):\, x \in \RR^p,\, z \in D(x)\right\}.
\end{align*}
$D$ is said to have {\em closed graph} or to be {\em graph closed} if $\mathrm{graph}\, D$ is closed as a subset of $\RR^{p + q}$. An equivalent characterization is that for any converging sequences $\left( x_k \right)_{k\in \NN}$, $\left( v_k \right)_{k\in \NN}$ in $\RR^p$, with $v_k \in D(x_k)$ for all $k \in \NN$, we have
\begin{align*}
				\lim_{k \to \infty} v_k \in D(\lim_{k\to\infty}x_k).
\end{align*}
$D$ is said to be locally bounded if for each compact $K \subset \RR^p$, there is $M > 0$ such that $\|v\|\leq M$ for all $v \in D(x)$ for all $x \in K$.
An {\em absolutely continuous curve} is a continuous function $x \colon \RR \to \RR^p$ which admits a derivative $\dot{x}$ for Lebesgue almost all $t \in \RR$, (in which case $\dot{x}$ is Lebesgue measurable), and $x(t) - x(0)$ is the Lebesgue integral of $\dot{x}$ between $0$ and $t$ for all $t \in \RR$. 

These elements allow to define the notion conservativity of set valued mappings \cite{bolte2020conservative}.
\begin{definition}[Conservative gradients]
			Let $D \colon \RR^p \rightrightarrows \RR^p$ be a set valued map with closed graph, non empty and locally bounded values and $f \colon \RR^p \to \RR$ a locally Lipschitz function. Then $f$ is a potential for $D$ if for all $x \in \RR^p$, all $\gamma \colon [0,1] \to \RR^p$, absolutely continuous with $\gamma(0) = 0$ and $\gamma(1) = x$, and all measurable functions, $v \colon [0,1] \to \RR^p$, such that $v(t) \in D(\gamma(t))$ for all $t \in [0,1]$,
			\begin{eqnarray}
			f(x)& = & f(0)+\int_0^1 \left\langle \dot{\gamma}(t), v(t) \right\rangle \dd t.
				\label{eq:conservative}
			\end{eqnarray}
			We shall also say that $D$ is a conservative gradient for~$f$ or simply a conservative gradient. Such functions $f$ are called path differentiable. The function $v$ in \eqref{eq:conservative} will be called a measurable selection of $D \circ \gamma$. 
		\label{def:conservativeMapForF}
\end{definition}
The result of \cite[Corollary 1]{bolte2020conservative} ensures that for a path differentiable $f$, $\partialc f$ is a conservative gradient. Note that in Definition \ref{def:conservativeMapForF}, if $f$ is not assumed to be locally Lipschitz a priori, the existence of a locally bounded conservative gradient ensures that it is locally Lipschitz, for instance by integration along segments.

\subsubsection{O-minimal structures}
Important references on this topic are \cite{Cos99,dries1996geometric}.
An {\em o-minimal structure} on $(\RR,+,\cdot)$ is a collection of sets
$\mathcal{O} = (\mathcal{O}_p)_{p \in \NN}$ where each $\mathcal{O}_p$ is itself a family of
subsets of $\RR^p$, such that for each $p \in \NN$:
\begin{enumerate}
\item[(i)] $\mathcal{O}_p$ is stable by complementation, finite union, finite intersection.
  \item[(ii)]  if $A$ belongs to $\mathcal{O}_p$, then both $A \times \RR$ and $\RR \times A$
    belong to $\mathcal{O}_{p+1}$;
  \item[(iii)]  if $\pi: \RR^{p+1} \to \RR^p$ is the canonical projection onto $\RR^p$ then,
    for any $A \in \mathcal{O}_{p+1}$, the set $\pi(A)$ belongs to $\mathcal{O}_p$;
    \label{it:algebraic}
  \item[(iv)]  $\mathcal{O}_p$ contains the family of real algebraic subsets of $\RR^p$, that is,
    every set of the form
    \[
      \{ x \in \RR^p \mid g(x) = 0 \}
    \]
    where $g: \RR^p \to \RR$ is a polynomial function;
  \item[(v)]  the elements of $\mathcal{O}_1$ are exactly the finite unions of  intervals.
\end{enumerate}

A subset of $\RR^p$ which belongs to an o-minimal structure $\mathcal{O}$ is said to be {\em definable in $\mathcal{O}$}. A function is {\em definable in $\mathcal{O}$} whenever its graph is definable in $\mathcal{O}$. A set valued mapping (or a function) is said to be definable in $\mathcal{O}$ whenever its graph is definable in $\mathcal{O}$. The terminology {\em tame} refers to definability in an o-minimal structure without specifying which structure. From now on we fix an o-minimal structure $\mathcal{O}$, definable sets being implicitly definable in $\mathcal{O}$.

The simplest  o-minimal structure is given by the class of real  semialgebraic objects.
Recall that a set $A \subset \RR^p$ is called {\em semialgebraic} if it is a finite union of sets of the form $$\displaystyle  \bigcap_{i=1}^k \{x \in \RR^p \mid g_{i}(x) < 0, \; h_{i}(x) = 0 \}$$
where the functions $g_{i}, h_{i}: \RR^p \to \RR$ are real polynomial functions and $k\geq 1$.
The key tool to show that these sets form an o-minimal structure is Tarski-Seidenberg principle which ensures that (iii) holds true.  As detailed in \cite{Cos99} this result can be expressed in the following way.
 \begin{proposition}[Quantifier elimination]
				 Any first order formula (quantification on variables only) involving polynomials, equalities and inequalities, definable functions and definable sets, describes a definable set.
							\label{prop:quantifierElimination}
 \end{proposition}

\subsection{Characterization of definable conservative gradients}
Our main convergence result holds under definability assumptions, we start by showing that conservativity admits a simpler characterization in this context.
From now on we fix an o-minimal structure (for example semialgebraic sets, see Section \ref{sec:presentation} for more details on definability), all definable objects we shall consider are implicitly definable in this structure. Recall that an o-minimal structure is a sequence of families of subsets $(\mathcal{O}_i)_{i \in \NN}$ such that for each $i \in \NN$, $\mathcal{O}_i$ contains subsets of $\RR^i$ and satisfy axiomatic compatibility conditions. 

\begin{definition}[Definably conservative gradient]
				Let $D \colon \RR^p \rightrightarrows \RR^{p}$ be a set valued field with closed graph, non empty and locally bounded values. 
				Assuming in addition that $D$ is definable, $D$ is called definably conservative if equation \eqref{eq:conservative} holds only for definable $C^1$ loops $\gamma$ and definable selections $v$. 
		\label{def:definablyConservative}
\end{definition}

Following Definition \ref{def:definablyConservative}, it is obvious that a definable conservative gradient is definably conservative since definable $C^1$ loops are absolutely continuous and definable selections are measurable. The following result provides a converse, a slightly more general proof is found in Section \ref{sec:definable}. 

\begin{theorem}
				Let $D \colon \RR^p \rightrightarrows \RR^{p}$ be definably conservative, then $D$ is conservative.
				\label{th:definablyConservative}
\end{theorem}

\begin{remark}
				It is of primary importance in Definition \ref{def:definablyConservative}, that the definable loops and definable set valued mapping $D$, are definable in the same o-minimal structure. For example consider the set $E$, the graph of the exponential function in $\RR^2$. By a theorem of Wilkie \cite{wilkie1999theorem}, there exists an o-minimal structure which contains all semi-algebraic sets and such that $E$ is definable in this structure, call it Wilkie's structure. Consider $D \colon \RR^2 \rightrightarrows \RR^2$ to be $\{0\}$ outside of $E$ and the unit Euclidean ball on $E$, it has a closed graph, it is bounded with nonempty values and definable in Wilkie's structure. Consider any differentiable semialgebraic loop $\gamma \colon [0,1] \to \RR^2$. Since $E$ is the graph of an analytic function, but is not a semialgebraic set, the intersection of $\gamma$ and $E$ must contain only finitely many points. Hence $\gamma$ and $D$ satisfy formula \eqref{eq:conservative}, but $D$ is obviously not conservative as it does not satisfy the integral formula along any nontrivial absolutely continuous path which image is in $E$. Hence the importance of having a unique fixed o-minimal structure throughout the manuscript. 
				\label{rem:sameStructure}
\end{remark}

\begin{remark}
				In the final stages of completing the first version of this manuscript, we became aware of the concurrent and independent work \cite{davis2021conservative}. Theorems 3.2 and 3.5 in \cite{davis2021conservative} proves equivalence between conservativity and definable conservativity provided that both $D$ and its associated potential $f$ are definable. The proposed result is very similar, it constitutes a minor improvement as it removes the requirement for definability of $f$ and only requires definability of $D$. The equivalence is a biproduct of the analysis of \cite{davis2021conservative} which contribution is orthogonal and in particular does not relates to parametric optimality and min-max problems.
				\label{rem:davis}
\end{remark}

This result shows that definable conservativity is equivalent to conservativity in the definable world. Its proof is based on an equivalent characterization of conservativity in this context, variational stratification \cite{bolte2007clarke,bolte2020conservative}.
The restriction to definable loops and selection in Definition \ref{def:definablyConservative} opens the possibility to use all results of o-minimal geometry \cite{dries1996geometric,Cos99} in order to prove that a given field is conservative. This can be in particular useful to prove conservativity of the PO formula.

\subsection{Definable parametric optimality formula}
Taking advantage of the strong rigidity of definable objects, we obtain the following result, which proof is stated in Section \ref{sec:application}.
\begin{theorem}
				Let $F \colon \RR^{p}\times \RR^{r} \to \RR$ be locally Lipschitz and definable. Set
				\begin{align*}
								f \colon \RR^p &\to \RR\\
								x &\mapsto \max_{y \in \RR^r} F(x,y),
				\end{align*}
				where the argmax is assumed to have nonempty and locally bounded values, call it $P(x) \subset \RR^r$. Set 
				\begin{align*}
								D_f \colon \RR^p &\rightrightarrows \RR^p\\
								x &\mapsto \conv\left\{ u,\  \exists y \in P(x),\,\, (u,0) \in \partialc F(x,y) \right\},
				\end{align*}
				then $D_f$ is conservative for $f$.
				\label{th:conservativePartialMinimization}
\end{theorem}
This result is the consequence of the more general Theorem \ref{th:definableSelection} stated in Example \ref{ex:partialMax}. The result is in fact stated for more general conservative gradients than the Clarke subdifferential, which opens the possibility to use objects defined through other calculus rules in place of subgradients, for example, outputs of algorithmic differentiation \cite{bolte2020conservative,bolte2020mathematical}.
This result implies that the PO formula can be used as a first order optimization oracle in the definable world, as illustrated in the next section. 

\subsection{Convergence of the ridge method}
Getting back to the initial problem
\begin{align*}
				\min_{x \in \RR^p} \max_{y \in \RR^r} F(x,y)
\end{align*}
where $F \colon \RR^p \times \RR^r \to \RR$ is locally Lipschitz and, in addition, definable. Assume that the mapping $x \rightrightarrows P(x) = \arg\max_{y \in \RR^r} F(x,y)$ is nonempty and locally bounded.
Consider the ridge algorithm, set $x_0 \in \RR^p$ and iterate for $k \in \NN$
\begin{align*}
				y_k &\in \arg\max_{y \in \RR^r} F(x_k,y)\\
				(u_k, 0) &\in \partialc F(x_k, y_k)\\
				x_{k+1} &= x_k - \alpha_k u_k
\end{align*}
The following results ensures under mild assumptions that accumulation points $(\bar{x}, \bar{y})$ of $(x_k, y_k)_{k\in \NN}$ are PO critical points for $f$ such that
\begin{align}
				\label{eq:POcritical}
				0 &\in \mathrm{conv} \{u,\, (u,0) \in \partialc F(\bar{x},y), \, y \in \arg\max_{z \in \RR^r} F(\bar{x},z)\} \\
				\bar{y} &\in \arg\max_{y \in \RR^r} F(\bar{x},y)\nonumber
\end{align}

\begin{theorem}
				Let $(\alpha_k)_{k\in\NN}$ be a non summable squence of positive step sizes tending to zero and assume that $(x_k)_{k \in \NN}$ is bounded. Then $F(x_k, y_k)$ converges, and all accumulation points $(\bar{x}, \bar{y})$ of $(x_k, y_k)_{k\in \NN}$ satisfy \eqref{eq:POcritical}.
				\label{th:convergence}
\end{theorem}
\begin{proof}
			This is a consequence of the results of \cite{benaim2005stochastic}, we sketch the proof arguments.

			The sequence defined by the \eqref{eq:ridgeAlgorithm} algorithm satisfies for all $k$, $u_k \in D_f(x_k)$ where $D_f$ is a convex valued conservative gradient for $f$ as described in Theorem \ref{th:conservativePartialMinimization}. Furthermore, condition \eqref{eq:POcritical} can be equivalently read as $0 \in D_f(x)$, see also \eqref{eq:optiCondition}. 
			
			Convex valued conservative gradients can be used in place of subgradients in first order method, while allowing to deploy the general method of \cite{benaim2005stochastic} for convergence analysis, we describe the main elements using the concepts and vocabulary in \cite{benaim2005stochastic}. 
			\begin{itemize}
							\item \cite[Proposition 1.3]{benaim2005stochastic} states that the recursion \eqref{eq:ridgeAlgorithm} produces a perturbed solution to the differential inclusion $\dot{x} \in -D_f(x)$. Indeed, we can apply definition III with a deterministic null perturbation term so that we meet step size requirements.
							\item \cite[Theorem 4.2]{benaim2005stochastic} ensures that perturbed solutions satisfy \cite[Theorem 4.1 (ii)]{benaim2005stochastic}: the curve $c: \RR^+ \to \RR^p$, given by piecewise affine interpolation of iterate sequence (\cite[Definition IV]{benaim2005stochastic}), is absolutely continuous, and for any $T > 0$, as $s \to \infty$, all accumulation points of $c_s \colon [0,T] \to \RR^p$, such that $c_s(t) = c(s+t)$ for all $s> 0$ and $t \in  [0,T]$, are solutions to the differential inclusion $\dot{x} \in -D_f(x)$.
							\item \cite[Theorem 4.3]{benaim2005stochastic} proves that limit sets of such curves are internally chain transitive with respect to the differential inclusion $\dot{x} \in -D_f(x)$. This is a precise notion of invariance given in \cite[Definition VI]{benaim2005stochastic}.
							\item Since $D_f$ conservative for $f$, $f$ is Lyapunov for the differential inclusion, and since both are definable, we have the Morse-Sard condition: $\left\{f(x), 0 \in D_f(x) \right\}$ is finite. Using Proposition 3.27 in \cite{benaim2005stochastic} we obtain that limit sets of the recurence is contained in PO critical solutions \eqref{eq:POcritical} and that $f$ is constant on this limit set, which is our result.
			\end{itemize}
\end{proof}

\begin{remark}
				The same result holds mutatis mutandis with a definable conservative gradient $D$ in place of $\partialc F$, for example one obtained by algorithmic differentiation \cite{bolte2020conservative}. Similarly minimization and maximization could be interchanged arbitrarily, modulo changes in the step sign.
				\label{rem:definableConservativeField}
\end{remark}

\begin{remark}
				The convex hull in \eqref{eq:POcritical} is necessary, for example we set $F(x,y) = xy - 2\vert\vert y\vert-1\vert$ for $x \in [-1,1]$, and extend by continuity, $F(x,y) = y - 2\vert\vert y\vert-1\vert$ for $x \geq 1$ and $F(x,y) = -y - 2\vert\vert y\vert-1\vert$ for $x\leq -1$. We have $f(x) = \max_{y} F(x,y) = \vert x\vert$ for all $x \in [-1,1]$ and $\max_{y} F(x,y) =1$ otherwise. In this case, the algorithm reduces to subgradient descent and the convex hull is necessary to obtain a valid optimality condition at $0$.
				\label{rem:convexHull}
\end{remark}

\subsection{Failure of parametric optimality formula in general}
It was already shown that the PO formula does not necessarily provide elements of the subdifferential (see footnote\footnotemark[\value{footnote}] on page \pageref{ex:failure1}). Yet the failure only occurred at the origin which does not prevent the PO formula to provide a conservative gradient for the value function $f$. The following result shows that this is not the case in general, its proof is given in Section \ref{sec:failure}.
\begin{theorem}
				There exists a Lipschitz path differentiable function $g \colon \RR^2 \to \RR$ such that
				\begin{itemize}
								\item for all $x \in [0,1]$, $\max_{y \in \RR} g(x,y) = x$ and the maximum is attained on $[0,1]$.
								\item for all $x \in [0,1]$, $0 \in \left\{ v,\, (v,0) \in \partialc g(x,y),\, y \in \arg\max_{y \in \RR}  g(x, y) \right\}$
								\item for countably many $x \in [0,1]$, $\arg\max_{y \in \RR} g(x,y)$ is a pair, for the rest it is a singleton.
				\end{itemize}
				\label{th:counterExample}
\end{theorem}
The preceding result shows that the PO formula does not provide a conservative gradient as the value function is the identity on $\RR$, but the formula may result in the constant $0$, which is not compatible with the integration constraint in Definition \ref{def:conservativeMapForF}. It is also obvious that the ridge algorithm \eqref{eq:ridgeAlgorithm} applied to minimization of $f \colon x \mapsto \max_{y \in [0,1]} g(x,y)$ based on the PO formula may get stuck at any initialization point $x \in [0,1]$ since they are all steady states of the algorithm \eqref{eq:ridgeAlgorithm}. This illustrates the fact that failure of conservativity entails in this case failure for first order algorithms based on the PO formula.
Obviously the function $g$ given in Theorem \ref{th:counterExample} is not definable in any structure as otherwise Theorem \ref{th:conservativePartialMinimization} would apply.

\section{Definably conservative gradients and parametric optimality formula}
\label{sec:definable}
The result described in Theorem \ref{th:definablyConservative} is actually proved for conservative jacobians, which generalize conservative gradients in a similar way as jacobians generalize gradients. We start with an extension of Definition \ref{def:conservativeMapForF} to this setting and make the necessary technical connections with the work of \cite{bolte2020conservative}, in particular the chain rule along absolutely continuous curves and the variational stratification. These preliminaries, although not explicitly stated this way in \cite{bolte2020conservative} are direct consequences of \cite{bolte2020conservative} and given here for completeness. We then proceed to the proof of the main results of this section: definably conservative jacobians are conservative and definable PO formula describes a conservative gradient.

\subsection{Conservative jacobians}

The following defines conservativity for matrix set valued functions using vanishing circulation. This was suggested in \cite{bolte2020conservative}, it is equivalent to \cite[Definition 4]{bolte2020conservative}.
\begin{definition}[Conservative jacobians]
				Let $J \colon \RR^p \rightrightarrows \RR^{m \times p}$ be a set valued map with closed graph, non empty and locally bounded values. 
				Then $J$ is a conservative jacobian, if for all $\gamma \colon [0,1] \to \RR^p$, absolutely continuous with $\gamma(0) = \gamma(1)$, and all measurable functions, $V \colon [0,1] \to \RR^{m \times p}$, such that $V(t) \in J(\gamma(t))$ for all $t \in [0,1]$,
				\begin{eqnarray}
								\int_0^1 V(t) \dot{\gamma}(t) \dd t = 0.
					\label{eq:conservativeMappings}
				\end{eqnarray}
				If in addition, $J$ is definable, then $J$ is called definably conservative if \eqref{eq:conservativeMappings} holds only for definable $C^1$ loops $\gamma$ and definable selections $V$. 
		\label{def:conservativeMapping}
\end{definition}
The following Lemma is a useful alternative characterization of conservativity through an operational chain rule corresponding to \cite[Definition 4]{bolte2020conservative}.
\begin{lemma}
				Let $J \colon \RR^p \rightrightarrows \RR^{m \times p}$, be a graph closed locally bounded non empty valued map. Then the following are equivalent
				\begin{itemize}
								\item[(i)] There exists $G \colon \RR^p \to \RR^m$, locally Lipschitz such that for any $\gamma \colon [0,1] \to \RR^p$ absolutely continuous, for almost all $t \in [0,1]$
												\begin{align*}
																\frac{d}{dt}G(\gamma(t)) =  M \dot{\gamma}(t), \qquad \forall M \in J(\gamma(t)).
												\end{align*}
								\item[(ii)] $J$ is a conservative jacobian as stated in Definition \ref{def:conservativeMapping}.
				\end{itemize}
				\label{lem:setValuedMapping}
\end{lemma}
In the situation of Lemma \ref{lem:setValuedMapping} (i), we will say that $J$ is a conservative jacobian for $G$, in which case $G$ is defined up to constants using line integrals. Similarly as in Definition \ref{def:conservativeMapForF}, Lipschicity of $G$ follows from the chain rule (i) in Lemma \ref{lem:setValuedMapping}.

\begin{remark}
		\label{rem:integralDerivativeDefinable}
		Lemma \ref{lem:setValuedMapping} holds \textit{mutatis mutandis} by taking $\gamma$ to be definable $C^1$ curve for definable conservativity in Definition \ref{def:conservativeMapping}.
\end{remark}

\subsection{Variational Stratification}
The notion of Variational Stratification was first exposed in \cite{bolte2007clarke} and found an interesting application to conservativity \cite{bolte2020conservative}.
The tangent space at a point $x$ of a differentiable manifold $M$ is denoted by $T_xM$\footnote{We only consider embedded manifolds and submanifolds}. Given a submanifold $M$ of a finite dimensional Riemannian manifold, it is endowed by the Riemanninan structure inherited from the ambient space. Given $G \colon \RR^p \to \RR^m$ and $M\subset\RR^p$ a differentiable submanifold  on which $G$ is differentiable, we denote by $\mathrm{Jac}_M\, G$ its Riemannian jacobian or even, when no confusion is possible,  $\jac G$.

 A $C^k$ stratification of a (sub)manifold $M$ (of $\RR^p$) is a partition $\SSS=(M_1,\ldots,M_m)$ of $M$ into $C^k$ manifolds having the property that $\cl M_i\cap M_j\neq \emptyset$ implies that $M_j$ is entirely contained in the boundary of $M_i$ whenever $i\neq j$. Assume that a function $G:M \to \RR^m$ is given and that $M$ is stratified into manifolds on which $G$ is differentiable. For $x$ in $M$, we denote by $M_x$ the strata containing $x$ and we simply write $\jac G(x)$ for the jacobian of $G$ with respect to $M_x$.

Stratifications can have many properties, we refer to \cite{dries1996geometric} and references therein for an account on this question and in particular for more on the idea  of a Whitney stratification. 
The definition is as follows: a $C^{r}$-stratification $\SSS=(M_{i})_{i\in I}$ of a manifold $M$ has the
\emph{Whitney-}($a$)\emph{ property,} if for each $x\in\cl M_{i}\cap
M_{j}$ (with $i\neq j$) and for each sequence $(x_{k})_{k\in\NN}\subset
M_{i}$ we have:
\[
\left.
\begin{array}
[c]{ll}
& \underset{k\rightarrow\infty}{\lim}\mathcal{\;}x_{k}\mathcal{\;}=x\\
& \text{}  \\
&
\underset{k\rightarrow\infty}{\lim}\mathcal{\;}T_{x_{k}}M_{i}\mathcal{\;}
=\mathcal{T}
\end{array}
\right\}  \mathcal{\;}\Longrightarrow\mathcal{\;}T_{x}M_{j}\mathcal{\;}
\subset\mathcal{\;T}
\]
where the second limit is to be understood in the Grassmanian, i.e., ``directional", sense. In the
sequel we shall use the term \emph{Whitney stratification} to refer to a
$C^{1}$-stratification with the Whitney-($a$) property. The following can be found for example in \cite[4.8]{dries1996geometric}.

\begin{theorem}[Whitney stratification]
    Let $A_1,\ldots, A_k$ be definable subsets of $\RR^p$, then there exists a definable Whitney stratification $(M_{i})_{i\in I}$ compatible with $A_1,\ldots, A_k$, \emph{i.e.} such that for each $i \in I$, and $t \in \left\{1 , \ldots k\right\}$, either $M_i \subset A_t$ or $M_i \cap A_t = \emptyset$.
    \label{th:wstratification}
\end{theorem}
For the rest of this section, $k$ denotes an arbitrary positive integer.
\begin{definition}[Variational stratification \cite{bolte2007clarke}]{\rm
								Let $G \colon \RR^p \to \RR^m$, be locally Lipschitz con\-ti\-nuous, let $J \colon \RR^p \rightrightarrows \RR^{m \times p}$ be a set valued map and let $k\geq 1$. We say that the couple $(G,J)$ has a $C^k$ {\em variational stratification} if there exists a $C^k$ Whitney stratification $\SSS = (M_i)_{i \in I}$ of $\RR^p$, such that $G$ is $C^k$ on each stratum  and for all $x \in \RR^p$,
				\begin{align}\label{projf}
								 J(x) \mathrm{Proj}_{T_{M_x}(x)} = \left\{ \jac G(x) \right\},
				\end{align}
				 where $\jac G(x)$ is the jacobian of $G$ restricted to the active strata $M_x$ containing $x$.	}	\label{def:projectionFormula}
\end{definition}

\begin{theorem}[Characterization of conservativity]
				Let $J \colon \RR^p \rightrightarrows \RR^{m \times p}$ be a definable, nonempty, locally bounded, graph closed set valued mapping and $G \colon \RR^p \to \RR^m$ be a definable locally Lipschitz function. Then the following are equivalent
    \begin{itemize}
        \item $J$ is conservative for $G$.
        \item $(G, J)$ admit a $C^k$ variational stratification.
    \end{itemize}
    For the reverse implication, $J$ and $G$ need not to be definable.
    \label{th:characterizationConservativity}
\end{theorem}
\begin{proof}
				This result is essentially known and we point out the arguments for completeness. First, $J$ is conservative for $G$ if and only if, the projection of each row of $J$ is conservative for the corresponding coordinate of $G$, this is Lemma 3 and 4 in \cite{bolte2020conservative} in combination with Lemma \ref{lem:setValuedMapping} above. Hence we may reason coordinatewise. 

				For the direct implication, it results from \cite[Theorem 4]{bolte2020conservative} that each coordinate of $G_i$ and the corresponding line of $J$, $J_i$ admit a variational projection formula (\cite[Definition 5]{bolte2020conservative}) for each $i=1,\ldots,m$, this corresponds to the variational formula introduced in \cite{bolte2007clarke}, which is limited to univariate outputs. The variational projection formula is stable by considering submanifolds and hence is stable when refining a given stratification. Hence thanks to Theorem \ref{th:wstratification}, we may find a common Whitney stratification such that the projection formula holds for each coordinate $G_i$ of $G$ and the corresponding row $J_i$ of $J$. This results in the formula given in Definition \ref{def:projectionFormula}.

				For the reverse implication, similarly as above, the variational stratification in Definition \ref{def:projectionFormula} implies the projection formula of \cite[Definition 5]{bolte2020conservative} for each coordinate of $G$ with the corresponding line of $J_i$. By \cite[Theorem 3]{bolte2020conservative} (see also \cite{davis2018stochastic} which states the result for the Clarke subdifferential), each row of $J$ is conservative for the corresponding coordinate of $G$ which implies that $J$ is conservative for $G$ by \cite[Lemma 4]{bolte2020conservative}. This does not require definability.
\end{proof}
\subsection{Definably conservative jacobians}
We start with a preliminary lemma which will then be applied recursively toward a proof of a variational stratification property from which Theorem \ref{th:definablyConservative} will follow.
\begin{lemma}
				Let $J \colon \RR^p \rightrightarrows \RR^{m \times p}$, be a definably conservative jacobian.
				Let $G \colon x \mapsto \int_0^1 V(\gamma(t)) \dot{\gamma}(t) dt$ for any $C^1$ definable $\gamma \colon [0,1] \to \RR^p$ with $\gamma(0) = 0$ and $\gamma(1) = x$ and any definable selection $V$ as in Definition \ref{def:conservativeMapping}.
				Then there exists a finite number of definable open sets $U_1,\ldots,U_N$ in $\RR^p$ such that $\cup_{i=1}^N \mathrm{cl}(U_i) = \RR^p$, $G$ is continuously differentiable on each set and $J = \{J_G\}$, the jacobian of $G$, on each set.
				\label{lem:setVlauedDefinable}
\end{lemma}
\begin{proof}
				Note that it is not known \textit{a priori} if $G$ is definable and therefore the results of \cite{bolte2020conservative,davis2021conservative} do not directly applies. Denote by $\mathcal{R}$ the set where $J$ is single valued, this set is definable by Proposition \ref{prop:quantifierElimination} and we are going to show that its complement has empty interior. Toward a contradiction, suppose that the complement of $\mathcal{R}$ has nonempty interior.  Then, definable choice \cite[4.5]{dries1996geometric} ensures that there exists two definable selections $V_1$ and $V_2$ such that $V_1 \neq V_2$ on a small open ball $B$. Let $v$ be a unit norm definable selection in $\mathrm{Im} (V_1 - V_2)^T$, which exists thanks to \cite[4.5]{dries1996geometric}. Since $v, V1, V2$ are definable, we may stratify using \cite[4.8]{dries1996geometric} so that $v,V1,V2$ can be assumed to be Lipschitz continuous on $B$ (reducing and translating $B$ if necessary). Call $r$ the radius of $B$ and assume without loss of generality that it is centered at $0$. Consider the solution to
				\begin{align*}
								\dot{\gamma}(t) = v(\gamma(t)),\qquad \gamma(0) = 0.
				\end{align*}
				$\gamma$ is $C^1$ and stays in a neighborhood of $0$ for small values of $t$, let's say that $\|\gamma(t)\| < r/2$ for all $t \in [-\alpha,\alpha]$ for a certain $\alpha > 0$. Let $0 < \epsilon < r/2$ be arbitrary and fix $\tilde{\gamma}$, a $C^1$ definable path with $\tilde{\gamma}(0) = 0$ such that
				\begin{align*}
								\max_{t \in\alpha[0,1]} \max\left\{\|\gamma(t) - \tilde{\gamma}(t)\|, \|\gamma'(t) - \tilde{\gamma}'(t)\|  \right\} \leq \epsilon,
				\end{align*}
				take for example a polynomial approximation of $\gamma'$ and its integral, using Weierstrass approximation Theorem.
				Since $\tilde{\gamma}$ is a $C^1$ definable arc which remains in $B$ by construction, for almost all $t$ around $0$, we have using Lemma \ref{lem:setValuedMapping} (see remark \ref{rem:integralDerivativeDefinable}),
				\begin{align*}
								\frac{d}{dt}G(\tilde{\gamma}(t)) = V_1(\tilde{\gamma}(t)) (v(\gamma(t)) + u(t)) = V_2(\tilde{\gamma}(t)) (v(\gamma(t)) + u(t))
				\end{align*}
				where $u(t) :=  \tilde{\gamma}'(t) - v(\gamma(t))$ for all $t$. By construction, we have $\|u(t)\| \leq \epsilon$ for all $t$. We can let $\epsilon \to 0$ along a sequence of such approximations $\tilde{\gamma}$, and using continuity of $V_1$ and $V_2$, we obtain for almost all $t$,
				\begin{align*}
								V_1(\gamma(t)) v(\gamma(t)) = V_2(\gamma(t)) v(\gamma(t))
				\end{align*}
				and therefore $(V_1(\gamma(t)) - V_2(\gamma(t)) v(\gamma(t)) = 0$, for all $t$ by continuity of $v,V_1,V_2$ and $\gamma$. Since $v(\gamma(t)) \in \mathrm{Im}(V_1(\gamma(t)) - V_2(\gamma(t))^T$, this shows that $v(\gamma(t)) = 0$ for all $t$ around $0$. This is contradictory with the fact that $v$ has unit norm. 

				This shows that the complement of $\mathcal{R}$ has empty interior. By stratification, using Theorem \ref{th:wstratification}, there exists $U_1,\ldots,U_N$  strata of maximal dimension such that the complement of $\mathcal{R}$ does not intersect any $U_i$, $i=1,\ldots,N$, and $\cup_{i=1}^N \mathrm{cl}(U_i) = \RR^p$.
				
				By graph closedness and local boundedness, $J$ can be identified with a continuous function on each $U_i$. Let $x \in U_i$ for some $i$, there is a small ball around $x$ such that $J$ is continuous on the ball. By the definition of $G$, for any $v \in \RR^p$ and $t > 0$ such that $x + tv$ remains in this ball,
				\begin{align*}
								\frac{G(x+tv) - G(x)}{t} &= \int_{0}^{1} J\left( x + stv \right) vds \\
								&= \frac{1}{t}\int_0^t J\left( x + sv \right) vds.
				\end{align*}
				Letting $t \to 0$, we have	
				\begin{align*}
								\lim_{t \to 0, t > 0} \frac{G(x+tv) - G(x)}{t} &= J(x)v
				\end{align*}
				where the limit is by continuity of $J$ at $x$. This formula allows to identify the partial derivatives at $x$ of each of the $m$ coordinate components of $G$ with entries of $J(x)$. Since these are continuous at $x$, $G$ is differentiable at $x$, and since $x \in U_i$ was arbitrary and $J$ is continuous on each $U_i$, $G$ is $C^1$ on each $U_i$ with jacobian $J$.
\end{proof}

\begin{theorem}
				Let $J \colon \RR^p \rightrightarrows \RR^{m \times p}$ be a definably conservative jacobian.
				Let $G \colon \RR^p \to \RR$ be defined as in Lemma \ref{lem:setVlauedDefinable}. Then $(G,J)$ admits a $C^k$ variational stratification: $(M_i)_{i\in I}$ a definable Whitney stratification of $\RR^p$ such that $G$ is $C^k$ on each stratum with
				\begin{align*}
								J(x) Q(x) = \{J_G(x)\},
				\end{align*}
				where $Q(x)$ is the matrix representing orthogonal projection to the tangent space of the active stratum $M(x)$ at $x$, seen as a subspace of $\RR^p$.
				\label{th:setValuedDefinable}
\end{theorem}
\begin{proof}
				We shall first prove that $(G, J)$ has a $C^1$ variational stratification. The $C^k$ variational stratification follows by definability of $J$ and the existence of $C^{k}$ stratification \cite[4.8]{dries1996geometric} for $J$ which allows to refine the constructed $C^1$ variational stratification to obtain differentiability up to order $k$. 

				Let $M$ be a connected definable $C^1$ submanifold embedded in $\RR^p$ which is also a connected $C^1$ cell (see \cite[4.2]{dries1996geometric}). Since $M$ is a connected $C^1$ cell, there is a definable $C^1$ diffeomorphism $\theta \colon M \to \RR^{\mathrm{dim} M}$, see for example \cite[Section 6.2]{Cos99}. Set $\phi \colon \RR^{\mathrm{dim} M} \to M$ such that $\phi = \theta^{-1}$. Consider now
				\begin{align*}
								\tilde{G} \colon \RR^{\mathrm{dim} M} &\to \RR^m \\
								\tilde{x} &\mapsto G (\phi(\tilde{x})) \\
								\tilde{J} \colon \RR^{\mathrm{dim} M} &\rightrightarrows  \RR^{m \times \dim M} \\
								\tilde{x} &\rightrightarrows J (\phi(x)) J_{\phi}(\tilde{x})
				\end{align*}
				For any definable $C^1$ path $\gamma \colon [0,1] \to \RR^{\mathrm{dim} M}$, 
				$\phi\circ \gamma$ is definable and so $(\tilde{G}, \tilde{J})$ satisfy the hypotheses of Lemma \ref{lem:setVlauedDefinable} since by Lemma \ref{lem:setValuedMapping} (see remark \ref{rem:integralDerivativeDefinable}), for almost all $t$  
				\begin{align*}
								\frac{d}{dt} \tilde{G}(\gamma (t)) = \frac{d}{dt} G(\phi \circ \gamma (t)) = J(\phi \circ \gamma(t)) (\phi \circ \gamma)'(t) =  J(\phi \circ \gamma(t)) J_{\phi}(\gamma(t)) \dot{\gamma}(t) = \tilde{J}(\gamma(t)) \dot{\gamma}(t),
				\end{align*}
				and $\gamma$ was an arbitrary $C^1$ definable path.

				Hence there exists $\tilde{U}_1 \ldots \tilde{U}_N$ open in $\RR^{\mathrm{dim} M}$ such that the union of their closure is equal the whole $\RR^{\mathrm{dim} M}$, and on each set $U_i$, $\tilde{J}$ is single valued and $\tilde{G}$ is $C^1$ with $\tilde{J}= \{J_{\tilde{G}}\}$. For $i = 1 \ldots N$, set $U_i = \phi(\tilde{U}_i)$ we have that the union of their relative closure is $M$: $\cup_{i=1}^N \mathrm{cl}_M(U_i) = M$. Furthermore, $G = \tilde{G} \circ \theta$ is differentiable on each $U_i$, relative to $M$, and
				\begin{align*}
								\{J_G(x)\} = \{J_{\tilde{G}} (\theta(x)) J_{\theta}(x)\} =  \{J (x) J_{\phi}(\theta(x)) J_{\theta}(x)\}.
				\end{align*}
				We remark that $J_{\phi}(\theta(x)) J_{\theta}(x)$ is the projection on the tangent space of $M$ at $x$ so that each line of $J$ has a single valued projection on the tangent space of $M$ on each $U_i$.
				
				Since $M$ was arbitrary, we may start with $M = \RR^p$ and proceed by induction on the dimension by applying Lemma \ref{lem:setVlauedDefinable} with the above reasoning. We have a $C^1$ variational projection on a dense set of strata (cells) of dimension $p$. We may obtain a Whitney stratification, compatible with this set, so that the projection formula does not hold only on a finite union of strata (cells) of dimension at most $p-1$ \cite[4.8]{dries1996geometric}. Each stratum $M$, being a $C^1$ embedded submanifold as well as a $C^1$ cell and we may repeat the process recursively until the dimension of the set where the projection formula does not hold is zero, i.e. a finite set of points, to obtain the desired Whitney stratification.
\end{proof}

The following corollary combines Theorems \ref{th:characterizationConservativity} and \ref{th:setValuedDefinable}. Theorem \ref{th:definablyConservative} is a special case for $m = 1$.
\begin{corollary}
				Let $J \colon \RR^p \rightrightarrows \RR^{m \times p}$ be a definably conservative jacobian, then $J$ is a conservative jacobian.
				\label{cor:definablyConservative}
\end{corollary}

\subsection{Application to PO formula}
\label{sec:application}
This section describe how Corollary \ref{cor:definablyConservative} can be used to prove that the PO formula provides a conservative gradient in the definable world. We start with a slightly more general result which generalizes the finite selection process described in \cite{bolte2020mathematical} from the discrete to the continuous setting.
\begin{theorem}
				Let $F \colon \RR^{p}\times \RR^{r} \to \RR$ be locally Lipschitz, $D$ a definable conservative jacobian for $F$ and $P \colon \RR^p \rightrightarrows \RR^p$ be a definable set valued field, with closed graph nonempty locally bounded values, such that for all $x \in \RR^p$ and $y \in P(x)$, there exists $u \in \RR^p$ such that $(u,0) \in D(x,y)$ and $y \mapsto F(x,y)$ is constant on $P(x)$. We set
				\begin{align*}
								f \colon x \mapsto F(x,P(x)).
				\end{align*}
				We have that $f$ is continuous, set
				\begin{align*}
								D_f \colon x &\mapsto \left\{ u,\  \exists y \in P(x),\,\, (u,0) \in D(x,y) \right\}.
				\end{align*}
				Then $D_f$ is conservative for $f$, and in particular, $f$ is locally Lipschitz.
				\label{th:definableSelection}
\end{theorem}

\begin{proof}
				One can check that $D_f$ is definable thanks to Proposition \ref{prop:quantifierElimination}. Furthermore, it has closed graph and is locally bounded with nonempty values.
				Hence, by Theorem \ref{th:definablyConservative} and Lemma \ref{lem:setValuedMapping}, we only have to prove a chain rule along definable $C^1$ curves. 

				Let $t \mapsto x(t) \in \RR^p$ be a $C^1$ definable path we will obtain definable selections of interest thanks to \cite[4.5]{dries1996geometric}, and draw conclusions thanks to Lemma \ref{lem:definableSelectors} which asserts that definable set valued fields have countable dense definable selections. 

				Let $t \to  y(t) \in\RR^p$ be a definable selection in $t \to P(x(t))$ and let $t \mapsto T(t) = (u(t), v(t))$ be a definable selection in $t \rightrightarrows D(x(t), y(t))$. Definable curves are piecewise continuously differentiable. Hence the functions $t \mapsto f(x(t)) = F(x(t), y(t))$, $t \mapsto x(t)$ and $t \mapsto y(t)$ are differentiable everywhere except at finitely many points, call them $t_1,\ldots, t_M$. Bounded definable curves have left and right limits everywhere, we have for $i=2,\ldots M$ that $t \mapsto y(t)$ is continuous on $(t_{i-1}, t_i)$ and can be extended to an absolutely continuous path on $[t_{i-1}, t_i]$. Hence we can use the fact that $D$ is conservative for $F$, which yields using graph closedness of $P$ and continuity of $F$, for $i=2,\ldots M$
				\begin{align*}
								f\left(x(t_i)\right) - f(x(t_{i-1})) &= F(x(t_i), \lim_{t \uparrow t_i} y(t))- F(x(t_{i-1}), \lim_{t \downarrow t_{i-1}} y(t)) \\
								&= \int_{t= t_{i-1}}^{t = t_i} \frac{d}{dt} F(x(t), y(t)) dt = \int_{t= t_{i-1}}^{t = t_i} \left\langle\dot{x}(t), u(t) \right\rangle + \left\langle\dot{y}(t), v(t) \right\rangle dt.
				\end{align*}
				Finally, by removing the set of discontinuity points of $y$, we obtain for almost all $t$.
				\begin{align*}
								\frac{d}{dt} f(x(t)) = \frac{d}{dt} F(x(t), y(t)) = \left\langle\dot{x}(t), u(t) \right\rangle + \left\langle\dot{y}(t), v(t) \right\rangle
				\end{align*}
				Since $T$ was an arbitrary definable selection of $t \rightrightarrows D(x(t), y(t))$, which admits a countable dense sequence of definable selections thanks to Lemma \ref{lem:definableSelectors}, we have for almost all $t$
				\begin{align*}
								\frac{d}{dt} f(x(t)) = \left\langle\dot{x}(t), u \right\rangle + \left\langle \dot{y}(t), v \right\rangle, \qquad \forall(u,v) \in D(x(t),y(t)).
				\end{align*}
				By the hypotheses, for each $u \in D_f(x)$, one can choose $v = 0$  and hence for almost all $t$
				\begin{align*}
								\frac{d}{dt} f(x(t)) = \left\langle\dot{x}(t), u \right\rangle  \qquad \forall u \text{ such that } (u,0) \in D(x(t),y(t)).
				\end{align*}
				Note that $y$ is an arbitrary definable selection in $t \to P(x(t))$. By Lemma \ref{lem:definableSelectors2}, there is a countable family of such selections $(y_i)_{i \in \NN}$, such that for all $t$, $$D_f(x(t)) = \mathrm{cl} \left\{ u,\, \exists i \in \NN, \, (u,0) \in D(x(t), y_i(t)) \right\}.$$ This implies that for almost all $t$ 
				\begin{align*}
								\frac{d}{dt} f(x(t)) = \left\langle\dot{x}(t), u \right\rangle  \qquad \forall u \in D_f(x(t)).
				\end{align*}
				Hence, we have a chain rule along definable arcs which is the desired result. Indeed, repeating the proof of direct implication in Lemma \ref{lem:setValuedMapping} for definable arcs, we have that $D_f$ is definably conservative and hence conservative thanks to Theorem \ref{th:definablyConservative} and admits $f$ as a potential.
\end{proof}

This result can be applied to partial maximization as the following example shows. It could also be applied to partial minimization or differentiation of more general critical values such as local minima or local maxima (under suitable assumptions). The following example is a repetition of Theorem \ref{th:conservativePartialMinimization} which is based on Theorem \ref{th:definableSelection}. 

\begin{example}[Partial maximization]
				Let $F \colon \RR^{p}\times \RR^{r} \to \RR$ be definable locally Lipschitz and $D \colon \RR^{p+r} \rightrightarrows \RR^{p+r}$ be definable and conservative for $F$ with convex values. Set
				\begin{align*}
								f \colon \RR^p &\to \RR\\
								x &\mapsto \max_{y \in \RR^r} F(x,y),
				\end{align*}
				where the argmax is assumed to be nonempty and locally bounded, call it $P(x) \subset \RR^r$. Set 
				\begin{align*}
								D_f \colon \RR^p &\rightrightarrows \RR^p\\
								x &\mapsto \conv\left\{ u,\  \exists y \in P(x),\,\, (u,0) \in D(x,y) \right\},
				\end{align*}
				one can use Theorem \ref{th:definableSelection} to show that $D_f$ is conservative for $f$.

				First $P$ is assumed to be have nonempty and locally bounded values, furthermore, it has a closed graph by continuity of $F$ and it is definable as partial maxima may be expressed using first order formula as in Proposition \ref{prop:quantifierElimination}. Second, by \cite[Example 10.12]{rockafellar1998variational}, for all $x \in \RR^p$ and $y \in P(x)$ there is $u \in \RR^p$ such that $(u,0) \in \partialc F^c(x,y) \subset D(x,y)$ (the inclusion follows from \cite[Corollary 1]{bolte2020conservative} and convexity of $D(x,y)$). Finally $y \mapsto F(x,y)$ is constant on $P(x)$ for all $x$ because the function value corresponds to the maximum, which is unique.

				This shows that $D$ and $P$ comply with Theorem \ref{th:definableSelection} and Theorem \ref{th:definableSelection} follows.
				\label{ex:partialMax}
\end{example}

\section{Failure of parametric optimality formula}
\label{sec:failure}
This section is dedicated to the construction of the function $g$ in Theorem \ref{th:counterExample}. We start with the construction of a fractal set $C$ and then describe the counterexample which will be based on the distance functions to $C$.
\subsection{A fractal set}
\label{sec:fractalSet}
Let $C$ be the fractal set which construction is described in Figure \ref{fig:fractal2}. The construction of $C$ is similar to the one described in \cite{whitney1935function} to provide a counterexample to Morse-Sard theorem. This construction was also used in \cite{rioszertuche2020pathological} to provide a subgradient sequence on a path differentiable function which fails to have dissipative and minimizing properties.
\begin{figure}[t]
				\centering
				\includegraphics[width=.8\textwidth]{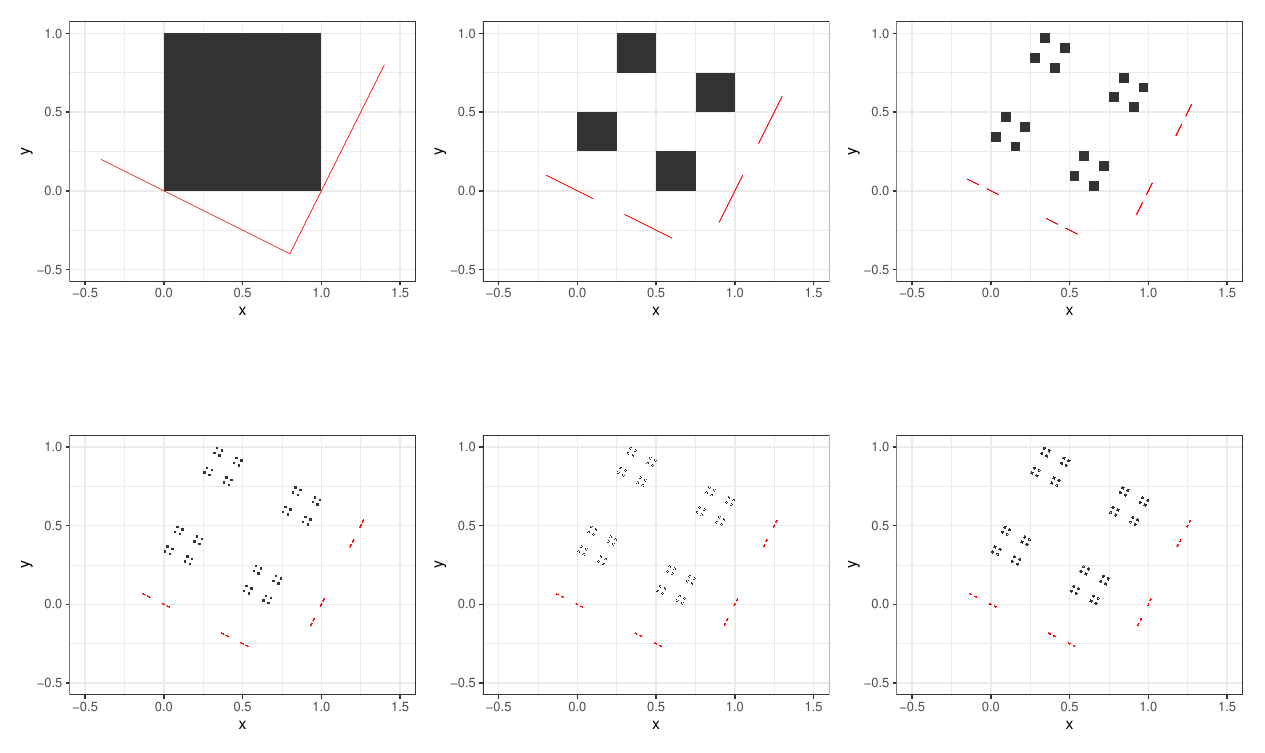}
				\caption{The fractal construction in Section \ref{sec:failure}. We start with the closed unit square in black. It is split into four copies of size one fourth the original square. This process is repeated recursively on each square ad infinitum. The additional red lines represent projection of these sets on rotated axes. Considering $C_i$, $i \in \NN$, the set obtained after $i$ steps ($C_0$ is the original square), we have $C_{i+1} \subset C_i$ for all $i$,  we let the limiting set to be $C = \cap_{i \in \NN} C_i$, which is closed. The projection of $C$ on each axes are full segments. Furthermore, in the limit, both projections on rotated axes are Cantor sets of zero measure.}
				\label{fig:fractal2}
\end{figure}

The set is defined as $C = \cap_{i \in \NN} C_i$, where for each $i \in \NN$, $C_i$ is the union of $4^i$ squares of size $1 / 4^i$. Furthermore, these sets form a nested decreasing sequence for the inclusion partial order. The set $C$ has the following properties.
\begin{itemize}
				\item $C$ is closed as an intersection of closed sets, hence it is compact.
				\item The projection of $C$ on each axis are full segments $[0,1] \times \{0\}$ and $\{0\} \times [0,1]$, we denote them by $\proj_x(C)$ and $\proj_y(C)$. Indeed, since we have a nested sequence $\proj_x(C) = \cap_{i \in \NN} \proj_x(C_i) = \cap_{i \in \NN}[0,1] = [0,1]$.
				\item In particular, $C$ is nonempty.
				\item The projection on each axis rotated counterclock-wise by an angle of $\arctan(2)$ is a Cantor set of zero measure (see Figure \ref{fig:fractal2}). Indeed, at each step there is a constant proportion of each segment which is removed from the projection, this is the simplest construction of Cantor sets.
				\item For each $i \geq 1$, there is a finite number of $x \in [0,1]$ for which the vertical line starting at $x$ intersects $C_i$ at two distinct squares, call this set $X_i$. We have $X_i \subset X_{i+1}$ and for each $x \in [0,1]$, $x \not \in X_i$, the intersection of the vertical line at $x$ and $C_i$ is at a single square. Set $X = \cup_{i \in \NN}X_i$, $X$ is denumerable and for each $x \in X$, the vertical line starting at $x$ intersects $C$ at two distinct points, for all other $x \in [0,1]$ this intersection is a singleton.
				\item Let $f \colon \RR \to \RR$ be a function such that for some $i \in \NN$, $(x,f(x)) \in C_i$ for all $x \in [0,1]$, then the total variation of $f$ is at least $i$. Hence if $f$ satisfies this property for all $i$, that is $(x,f(x)) \in C$ for all $x \in [0,1]$, then $f$ has infinite total variation.
\end{itemize}

\subsection{Construction and proof of the counterexample}
Consider the following Lipschitz function
\begin{align*}
				f \colon \RR^2 &\to \RR \\
				z &\mapsto - \dist\left( z, C \right).
\end{align*}
Set $O = \RR^2 \setminus C$, $O$ is open and we have for all $z \in O$
\begin{align*}
				f(z) = - \min_{c \in C}	\|z - c\| = \max_{c \in C} -\|z-c\|.
\end{align*}
Each function $f_c \colon z \mapsto -\|z-c\|$ is $C^1$ on $O$ and both $f_c$ and $\nabla f_c$ are jointly continuous with respect to $z$ and $c$ on $O \times C$. This shows that $f$ is lower $C^1$ on $O$ \cite[Definition 10.29]{rockafellar1998variational}. Hence we have that $f$ is subdifferentially regular on $O$ \cite[Theorem 10.31]{rockafellar1998variational}. Combining with Lemma \ref{lem:projZeroMeasure} we have
\begin{lemma}
				$f$ is path differentiable.
				\label{lem:fPathDifferentiable}
\end{lemma}

\begin{proof}
				We will prove that $\partialc f$ satisfies the chain rule along absolutely continuous curves \cite{davis2018stochastic,bolte2020conservative}, since $f$ is Lipschitz, this is sufficient to conclude, see also Lemma \ref{lem:setValuedMapping}.
				Let $\gamma \colon [0,1] \to \RR^2$ be an absolutely continuous curve and $R \subset [0,1]$ be the full measure set where $\gamma$ and $f \circ \gamma$ are differentiable, we will show that $\gamma$ satisfies the chain rule for almost all $t \in R$ which is sufficient to conclude. We set
				\begin{align*}
								E = \left\{ t \in R,\, \gamma(t) \not\in C \right\} 
				\end{align*}
				We also consider $\tilde{E} \subset E$ with
				\begin{align*}
								\tilde{E} = \left\{ t \in E,\, \max_{v \in \partialc f(\gamma(t))} \left\vert\left\langle v, \dot{\gamma}(t)\right\rangle - \frac{d}{dt} f(\gamma(t))\right\vert > 0\right\}.
				\end{align*}

				Fix $t \in \tilde{E}$ arbitrary, since $t \in E$, choosing $a$ small enough, we have by continuity of $\gamma$, $\gamma([t - a, t + a]) \cap C = \emptyset$ and  $\gamma([t - a, t + a]) \subset O$. Since $f$ is lower $C^1$ on $O$, by \cite[Theorem 10.31]{rockafellar1998variational}, $f$ is subdifferentially regular in a neighborhood of $\gamma([t-a, t + a]) \subset O$. We may apply \cite[Lemma 5.4]{davis2018stochastic} which shows that $f$ satisfies the chain rule along the curve $\gamma$ restricted to the closed segment $I = [t-a, t + a]$, in other words $[t-a, t + a]\cap \tilde{E}$ has measure zero. The segment can be taken of arbitrarily small length, therefore, such intervals form a Vitali covering set (see for example \cite[Section 6.2]{royden2010real}). By the Vitali covering theorem \cite[Section 6.2]{royden2010real} for any $\epsilon > 0$, there is a finite collection of such segments $I_1,\ldots,I_K$ such that $\tilde{E} \setminus \cup_{k=1}^K I_k$ has (outer) measure at most $\epsilon$ which shows that $\tilde{E}$ has (outer) measure arbitrary small and therefore has measure zero.


				Now set 
				\begin{align*}
								E_1 &= \left\{ t \in R,\, \gamma(t) \in C,\, \gamma'(t) \neq 0 \right\} \\
								E_2 &= \left\{ t \in R,\, \gamma(t) \in C,\, \gamma'(t) = 0 \right\}.
				\end{align*}
				We have $R = E \cup E_1 \cup E_2$, we have shown that the chain rule holds for almost all $t$ in $E$, Lemma \ref{lem:projZeroMeasure2} ensures that $E_1$ has zero measure and the chain rule holds trivially for all $t \in E_2$ because $f$ is Lipschitz. This shows that the chain rule holds for almost all $t \in \RR$ and the conclusion follows.
\end{proof}

\begin{figure}[t]
				\centering
				\includegraphics[width=.8\textwidth]{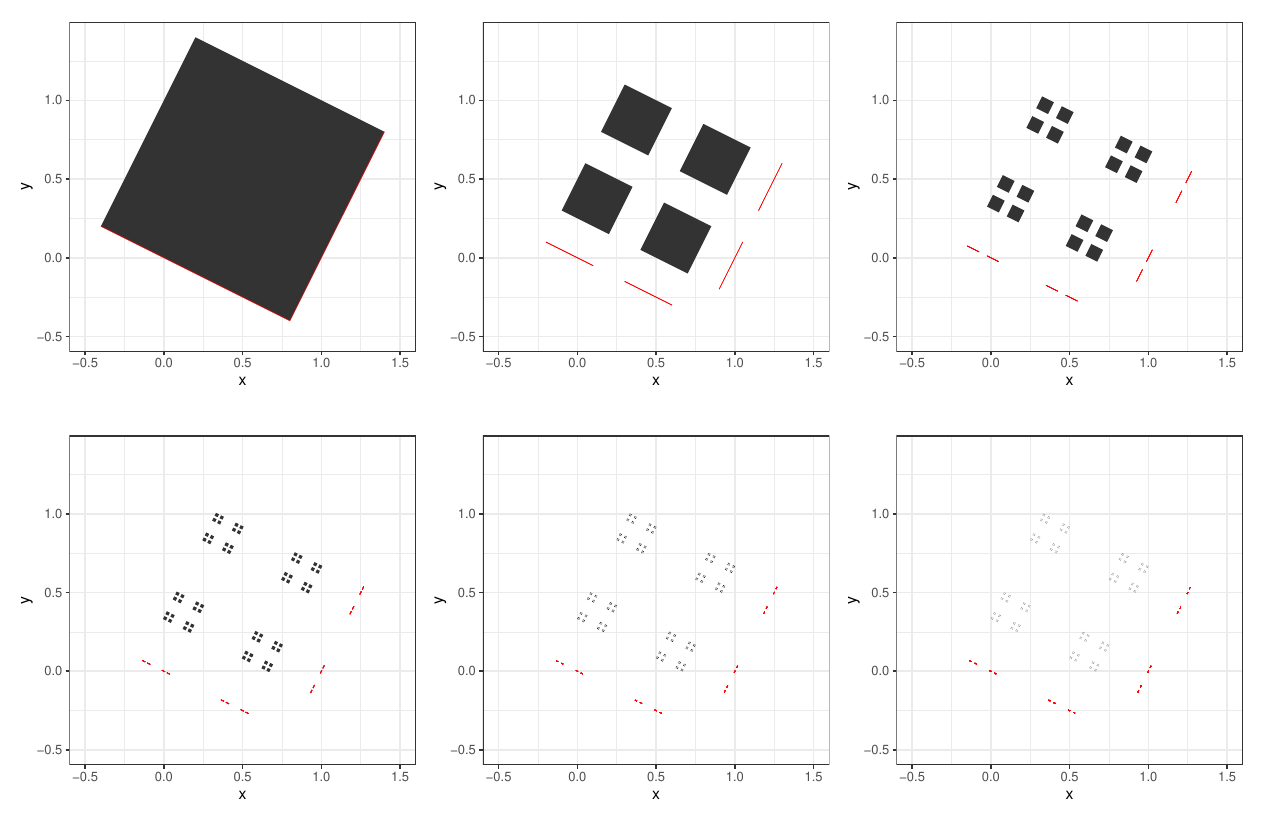}
				\caption{Same construction as in Figure \ref{fig:fractal2}, except that we keep the whole product set, the red line representing the projections on rotated axes and the black squares representing their products. In the limit, we obtain the same set $C$ as in Figure \ref{fig:fractal2} since the distance between the set represented here and those of Figure \ref{fig:fractal2} goes to zero.}
				\label{fig:fractal3}
\end{figure}
Now, we should characterize the subdifferential of $f$, which directly relates to the normal cone to $C$. We will use the notion of normal cone as described in \cite[Definition 6.3]{rockafellar1998variational}. As depicted in Figure \ref{fig:fractal3}, it turns out that $C$ is actually a product of Cantor set which are closed and have empty interior.
Using Lemma \ref{lem:cantorNormalCone} we obtain by \cite[Proposition 6.41]{rockafellar1998variational} that the normal cone to $C$ is $\RR^2$ everywhere on $C$. We deduce the following.
\begin{lemma}
				For all $z \in C$, $\partialc f(z) = B$, where $B$ is the unit ball.
				\label{lem:clarkeDistanceFun}
\end{lemma}
\begin{proof}
				As shown in Figure \ref{fig:fractal2}, $C$ is contained in a product of cantor sets, $\tilde{C}$. Actually $C$ is equal to a product of Cantor sets. Indeed, as detailed in Figure \ref{fig:fractal3}, the distance between the sets constructed in Figure \ref{fig:fractal2} and those of Figure \ref{fig:fractal3} decreases to $0$, so that limiting intersections are the same. By \cite[Proposition 6.41]{rockafellar1998variational} and using Lemma \ref{lem:cantorNormalCone}, $N_{\tilde{C}} = \RR^2$ everywhere on $\tilde{C}$ hence on $C$. The result follows from \cite[Example 8.53]{rockafellar1998variational} using the fact that the Clarke subdifferential is the convex closure of the limiting subdifferential and it commutes with multiplication by scalars.
\end{proof}

For the function $g$ in Theorem \ref{th:counterExample}, one could take for example $g\colon (x,y) \mapsto 2 f(x,y) + x$ which satisfies all the required properties
\begin{itemize}
				\item Since the projection of $C$ on $x$ is $[0,1]$ (Section \ref{sec:fractalSet}), we have $\max_{y \in [0,1]} g(x,y) = x$.
				\item Since for all but countably many $x$, the intersection of $C$ with the vertical line at $x$ is a singleton (Section \ref{sec:fractalSet}), for such $x$, the argmax is unique, for the others it is a pair.
				\item The sum rule holds for $f$, so for every $x \in [0,1]$ and $y \in \arg\max_{y \in [0,1]} g(x,y)$, we have $(x,y) \in C$ so $\partialc f(x,y) = B$, by Lemma \ref{lem:clarkeDistanceFun}. By sum rule, $\partialc g(x,y) = 2 B + (1,0)$ so that 
				\begin{align*}
								\left\{ v,\, (v,0) \in \partialc g(x,y),\, y \in \arg\max_{y \in [0,1]}  g(x, y) \right\} = [-1,3]
				\end{align*}
				which contains $0$.
\end{itemize}

\subsection{Need for a better behaved subclass}
\label{sec:needDiscussion}
In the definition of the function $g$, the fractal nature of the construction makes the argmin mapping, $P \colon \RR^p \rightrightarrows \RR^r$, although almost everywhere a singleton, highly non regular. In this example, it is not even of bounded variation in the sense that it is not possible to obtain bounded variation selections in $P$, a fortiori not absolutely continuous. This explains why conservativity is destroyed, the connection between $g$, $F$ and its value function $f$, through $P$ in the PO formula  takes place outside of the absolutely continuous world and hence outside of the conservative world which is built on absolutely continuous paths.  

Therefore additional restrictions on the function $F$ have to be enforced if one wants a calculus rule for the PO formula which preserves conservativity. An intuitive direction, is to ensure that the argmax mapping admits selections which are absolutely continuous, or close to absolute continuity, in order to apply and use the definition of conservativity in Definition \ref{def:conservativeMapForF}. There would be potentially many possible such classes, in Section \ref{sec:application} we focus on one of them, definable functions \cite{dries1996geometric,Cos99}, for which we have access to definable selections which are piecewise differentiable, which is close enough to absolute continuity for our purpose.

\section{Lemmas and proofs}
\label{sec:additionalLemmas}
\begin{proof}[Proof of Lemma \ref{lem:setValuedMapping}]
				The fact that (i) implies (ii) is direct by integration. Indeed,  $G\circ \gamma$ is absolutely continuous. For any $\gamma \colon [0,1] \to \RR^p$ absolutely continuous, with $\gamma(0) = \gamma(1)$, and measurable selection $V$ as in the statement of the Lemma,
				\begin{align*}
								\int_0^1 V(\gamma(t)) \dot{\gamma}(t) dt = \int_0^1 \frac{d}{dt}G(\gamma(t)) dt = G(\gamma(1)) - G(\gamma(0)) = 0,
				\end{align*}
				where the first inequality uses (i) and the fact that $V(\gamma(t)) \in J(t)$ for all $t$ and the second is absolute continuity of $G \circ \gamma$.

				For the reverse implication, fix $x \in \RR^p$, an absolutely continuous path $\gamma \colon [0,1] \to \RR^p$, $\gamma(0) = 0$, $\gamma(1) = x$ and a measurable selection $V \colon \RR^p \to \RR^p$ such that for all $z \in \RR^p$, $V \in J(z)$. We define 
				\begin{align*}
								G(x) = \int_0^1 V(t) \dot{\gamma}(t) dt.
				\end{align*}
				By (ii), the value of $G$ does not depend on the choice of $\gamma$ and on the measurable selection $V$. Furthermore, we have by Lebesgue differentiation theorem, for almost all $t \in [0,1]$, 
				\begin{align}
								\frac{d}{dt}G(\gamma(t)) =  V(t) \dot{\gamma}(t)
								\label{eq:lemSetValued1}
				\end{align}
				Now since $J$ is nonempty compact valued, from \cite[corollary 18.15]{aliprantis2005infinite}, there exists a sequence $(V_i)_{i \in \NN}$ of measurable selections such that for all $t \in [0,1]$
				\begin{align}
								J(\gamma(t)) = \mathrm{cl} \{V_i(\gamma(t))\}_{i \in \NN}.
								\label{eq:lemSetValued2}
				\end{align}
				The result follows by combining \eqref{eq:lemSetValued1} and \eqref{eq:lemSetValued2}.
\end{proof}

The following Lemma is essentially a repetition of \cite[Lemma 16]{rioszertuche2020pathological} which we reproduce here for completeness.
\begin{lemma}
				Let $C \subset \RR^2$ be a closed set which projections on the $x$ and $y$ axes have measure zero respectively. Then for any Lipschitz curve, $\gamma \colon \RR \to \RR^2$, the set
				\begin{align*}
								E = \left\{ t \in \RR, \gamma(t) \in C \right\} \setminus \left\{ t \in \RR, \gamma'(t) = 0 \right\}
				\end{align*}
				has measure zero.
				\label{lem:projZeroMeasure}
\end{lemma}
\begin{proof}
				Write $\gamma_1$, $\gamma_2$ the coordinates of $\gamma$ and $P_1$, $P_2$ the projection of $C$ on the $x$ and $y$ axes respectively, by hypothesis they have measure zero. Set $B\subset \RR$ the set where $\gamma'$ is either not defined or well defined and different from zero. We have $E = \gamma^{-1}(C) \cap B$.

				Set $A_1 \subset B$ and $A_2\subset B$ the sets where $\gamma'_1 \neq 0$ and $\gamma'_2 \neq 0$ respectively. Set $A_3 \subset B$ the zero measure set where $\gamma'$ is not defined. We have that, $A_1 \cup A_2 \cup A_3= B$.

				Consider an enumeration of intervals $(I_j)_{j\in \NN}$ of the form $[p_j, q_j]$, with $p_j,q_j \in \QQ$ and $p_j < q_j$ for $j \in \NN$. Fix $j \in \NN$ and $i = 1,2$ and consider the following function on $I_j$
				\begin{align*}
								f_j^i \colon t \rightarrow \min_{s \in I_j,\, \gamma_i(s) = \gamma_i(t)}s.
				\end{align*}
				The set valued function $t \rightrightarrows \left\{ s \in I_j,\, \gamma_i(s) = \gamma_i(t) \right\}$ has closed graph and nonempty compact values on $I_j$ and hence $f_j^i$ is measurable by \cite[Theorems 18.19 and 18.20]{aliprantis2005infinite}. Set
				\begin{align*}
								Q_j^i = \left\{ t \in I_j, t = f_j^i(t) \right\} \cap A_i \cap \gamma_i^{-1}(P_i).
				\end{align*}
				$Q_j^i$ is measurable and $\gamma_i$ is injective on $Q_i^j$ by construction and $\gamma_i(Q_j^i) \subset P_i$ .

				Using the injectivity of $\gamma_i$ on $Q_j^i$ and a change of variable formula \cite[Theorem 3.8]{evans2015measure}, we have for any $j \in \NN$ and $i = 1,2$,
				\begin{align*}
								0 \leq \int_{Q_j^i} \vert\gamma_i'(t)\vert dt = \int_{\gamma_i(Q_j^i)} 1 dt \leq \int_{P_i} 1 dt	 = 0
				\end{align*}
				Since  $\vert\gamma_i'(t)\vert>0$ on $Q_j^i$, for any $j \in \NN$ and $i = 1,2$, this means that $Q_j^i$ has measure zero.

				Now consider $t \in \gamma_i^{-1}(P_i) \cap A_i$, this means that there exists $p \in P_i$ such that $t \in \gamma_i^{-1}(p) \cap A_i$. Since $t \in A_i$ this means that $\gamma'_i(t) \neq 0$ and so there exists an interval $I_j$ containing $t$, such that $\{t\} = \left\{ s \in I_j, \gamma_i(s) = \gamma_i(t) \right\}$. This shows that $t \in Q_j^i$ and since $t$ was arbitrary, we have that $\gamma_i^{-1}(P_i) \cap A_i = \cup_{j \in \NN} Q_j^i$ and hence $\gamma_i^{-1}(P_i) \cap A_i$ has zero measure as the countable union of zero measure sets.

				We have
				\begin{align*}
								E = \gamma^{-1}(C) \cap B &= \gamma^{-1}(C) \cap \left( A_1 \cup A_2 \cup A_3 \right) \\
								&\subset \gamma_1^{-1}(P_1) \cap \gamma_2^{-1}(P_2) \cap \left( A_1 \cup A_2 \cup A_3 \right)\\
								&\subset (\gamma_1^{-1}(P_1) \cap A_1) \cup (\gamma_2^{-1}(P_2) \cap A_2) \cup A_3.
				\end{align*}
				All three sets on the right hand side have zero measures so $E$ has zero measure.
\end{proof}

\begin{lemma}
				The result of Lemma \ref{lem:projZeroMeasure} holds for any absolutely continuous curve $\gamma \colon [0,1] \to \RR^2$.
				\label{lem:projZeroMeasure2}
\end{lemma}
\begin{proof}
				From \cite[Lemma 1.1.4]{ambrosio2008gradient}, $\gamma$ admits a Lipschitz reparametrization. That is, there exists an increasing absolutely continuous function $s \colon \RR \to \RR$ with Lipschitz inverse $t$ and a Lipschitz curve $\hat{\gamma}$, such that
				\begin{align}
								\hat{\gamma} &= \gamma \circ t \nonumber\\
								\|\hat{\gamma}'\| \circ s &= \frac{\|\gamma'\|}{1 +  \|\gamma'\|} 
								\label{eq:ambrosio}
				\end{align}
				where the second identity holds almost everywhere. Lemma \ref{lem:projZeroMeasure2} holds true for $\hat{\gamma}$. We have for almost all $a$ using \eqref{eq:ambrosio} and the fact that $t$ is the inverse of $s$,
				\begin{align*}
								& \hat{\gamma}(a) \in C,\, \|\hat{\gamma}'\|(a) \neq 0 \\
								\Leftrightarrow \qquad& \gamma(t(a)) \in C,\, \|\hat{\gamma}'\| \circ s (t(a)) \neq 0 \\ 
								\Leftrightarrow \qquad& \gamma(t(a)) \in C,\, \|\gamma'(t(a))\|  \neq 0 .
				\end{align*}
				Set $A$ the full measure set where the previous equivalence holds, we have 
				\begin{align*}
								\left\{t,\, \gamma(t) \in C,\, \|\gamma'(t)\|  \neq 0 \right\} \subset t\left( \left\{ a,\, \hat{\gamma}(a) \in C,\, \|\hat{\gamma}'\|(a) \neq 0 \right\} \right) \cup t(A^c).
				\end{align*}
				The set $\left\{ s,\, \hat{\gamma}(s) \in C,\, \|\hat{\gamma}'\|(s) \neq 0 \right\}$ has measure zero by Lemma \ref{lem:projZeroMeasure2} because $\hat{\gamma}$ is Lipschitz. 
				The Lebesgue measure of the image of a zero measure set by a Lipschitz map is zero, and, since $t$ is Lipschitz, the right hand side has measure zero. This is the desired result.
\end{proof}

\begin{lemma}
				Let $C \subset \RR$ be a closed set with empty interior. Then for all $x \in C$, $N_C(x) = \RR$.
				\label{lem:cantorNormalCone}
\end{lemma}
\begin{proof}
				Denote by $f$ the distance function to $C$. $f$ is $1$-Lipschitz and hence differentiable almost everywhere. Furthermore, $f(x) = 0$ if and only if $x \in C$ since $C$ is closed. Fix $x \in C$ and construct a sequence $(x_k)_{k \in \NN}$ as follows, choose $\epsilon_k = 1 / (k+1)$ for $k \in \NN$, 
				\begin{itemize}
								\item If $(x,x + \epsilon_k) \cap C = \emptyset$, then $x_k = x$ and we have that $1 \in N_C(x_k)$.
								\item Otherwise choose $z_k \in (x,x + \epsilon_k) \cap C$ and $y_k = \arg\max_{y \in [x, z_k]} f(y)$ where the maximum is positive by closedness of $C$ and the fact that $C$ has empty interior. $y_k$ has a projection on the left and on the right on $C$ (otherwise it cannot be in the argmax). Choose $x_k$ to be the projection on the left. We have that $1 \in N_C(x_k)$, $x \leq x_k \leq x + \epsilon_k$.
				\end{itemize}
				In all cases we have $x_k \to x$ and $1 \in N_C(x_k)$ which shows that $1 \in N_C(x)$. Similarly one could show that $-1 \in N_C(x)$ and hence $N_C(x) = \RR$ since it is a cone. Since $x$ was arbitrary in $C$, this proves the desired result.
\end{proof}

\begin{lemma}
				Let $J \colon \RR^p \to \RR^{m}$, be a definable compact valued map with nonempty values. Then there exists a sequence of definable selections $(V_i)_{i \in \NN}$ for $J$ such that for all $x \in \RR^p$
				\begin{align*}
								J(x) = \mathrm{cl} \{V_i(x)\}_{i \in \NN}.
				\end{align*}
				\label{lem:definableSelectors}
\end{lemma}

\begin{proof}
				Let $V_1$ be any definable selection of $J$, such a $V_1$ exists by \cite[4.5]{dries1996geometric}. Set by recursion, for $i = 2,\ldots$, $V_i$ a definable selection of
				\begin{align*}
								x \rightrightarrows \arg\max_{v \in J(x)} \mathrm{dist}\left( v, \left\{ V_j(x) \right\}_{j=1}^{i-1} \right).
				\end{align*}
				which is definable. By a simple covering argument, using compacity, for all $x$
				\begin{align*}
								\lim_{i \to \infty} \max_{v \in J(x)} \mathrm{dist}\left( v, \left\{ V_j(x) \right\}_{j=1}^{i-1} \right) = 0,
				\end{align*}
				which shows that the constructed sequence has the desired property.
\end{proof}

\begin{lemma}
				Using the notations of Theorem \ref{th:definableSelection}, there is a sequence of definable selections, $(y_i)_{i \in\NN}$ such that for all $i \in \NN$, and all $x \in \RR^p$, $y_i(x) \in P(x)$ and 
				\begin{align*}
								D_f(x) = \left\{ u,\  \exists y \in P(x),\,\, (u,0) \in D(x,y) \right\} = \mathrm{cl} \left\{ u,\   (u,0) \in D(x,y_i(x)) \right\}.
				\end{align*}
				\label{lem:definableSelectors2}
\end{lemma}
\begin{proof}
				Note that $D_f$ has compact values and is definable by Proposition \ref{prop:quantifierElimination}. Using Lemma \ref{lem:definableSelectors} we have a definable sequence $(u_i)_{i \in \NN}$ of selections in $D_f$ such that $D_f(x) = \mathrm{cl} \{u_i(x)\}_{i \in \NN}$. By definability, for $i \in \NN$ using \cite[4.5]{dries1996geometric}, we can choose a definable sequence $(y_i)_{i \in \NN}$ such that for all $x \in \RR^p$,
				\begin{align*}
								y_i(x) &\in P(x)\\
								(u_i(x),0) &\in D(x,y_i(x)). 
				\end{align*}
				We have for all $x \in \RR^p$
				\begin{align*}
								\lim_{i \to \infty} \max_{u \in D_f(x)} \mathrm{dist}\left( u, \left\{ \tilde{u},\, (\tilde{u},0) \in D(x,y_j(x))  \right\}_{j=1}^{i-1} \right) \leq \lim_{i \to \infty} \max_{u \in D_f(x)} \mathrm{dist}\left( u, \left\{ u_j(x)\right\}_{j=1}^{i-1} \right) =  0,
				\end{align*}
				which shows that the constructed sequence has the desired property.

\end{proof}

\bigskip

\noindent
{\bf Acknowledgments.} 
The author would like to thank J\'er\^ome Bolte and Rodolfo Rios-Zeruche for interesting discussions which helped putting this work together.
The author acknowledge the support of ANR-3IA Artificial and Natural Intelligence Toulouse Institute under the grant agreement ANR-19-PI3A-0004,
Air Force Office of Scientific Research, Air Force Material Command, USAF, under grant numbers FA9550-19-1-7026, FA8655-22-1-7012, 
and ANR MaSDOL - 19-CE23-0017-01.

\end{document}